\documentclass[a4paper,11pt]{amsart}

\usepackage[T1]{fontenc}
\usepackage[bottom=3.5cm, right=3cm, left=3cm, top=3.5cm]{geometry}
\usepackage{amsmath,amssymb,amscd,amsfonts}
\usepackage{hyperref}
\usepackage{graphicx}
\usepackage{enumitem}
\usepackage{paralist}
\usepackage{csquotes}
\usepackage{cleveref}
\usepackage{tikz-cd}

\newtheorem{theorem}{Theorem}[section]
\newtheorem{proposition}[theorem]{Proposition}
\newtheorem{lemma}[theorem]{Lemma}
\newtheorem{corollary}[theorem]{Corollary}

\AddToHook{env/lemma/begin}{\crefalias{theorem}{lemma}}
\crefname{lemma}{lemma}{lemmas}
\AddToHook{env/corollary/begin}{\crefalias{theorem}{corollary}}
\crefname{corollary}{corollary}{corollary}
\AddToHook{env/proposition/begin}{\crefalias{theorem}{proposition}}
\crefname{proposition}{proposition}{proposition}
\AddToHook{env/remark/begin}{\crefalias{theorem}{remark}}
\crefname{remark}{remark}{remark}

\theoremstyle{definition}
\newtheorem{definition}[theorem]{Definition}
\theoremstyle{remark}
\newtheorem{remark}[theorem]{Remark}

\renewcommand{\AA}{\mathbb A}
\newcommand{\PP}{\mathbb P}

\newcommand{\RR}{\mathbb R}
\newcommand{\ZZ}{\mathbb Z}
\newcommand{\FF}{\mathbb F}
\newcommand{\RP}{\mathbb{RP}}

\newcommand{\GW}{\operatorname{GW}}
\newcommand{\W}{\operatorname{W}}

\newcommand{\im}{\operatorname{im}}
\newcommand{\rk}{\operatorname{rk}}
\newcommand{\codim}{\operatorname{codim}}
\newcommand{\chara}{\operatorname{char}}
\newcommand{\colim}{\operatorname{colim}}

\begin{document}

\title{Real cycle class isomorphism for linear schemes}
\author{Jan Hennig}

\begin{abstract}
The real cycle class map \[H^i(X,\underline{I}^j(\mathcal{L})) \rightarrow H^i_\text{sing}(X(\RR),\ZZ(\mathcal{L}))\] is an isomorphism for $j\geq \dim(X)+1$ for any scheme $X$ over $\RR$ by a result of Jacobson. It is also known to be an isomorphism for $j\geq i$, the earliest possible case, if $X$ is cellular due to Hornbostel--Wendt--Xie--Zibrowius. This paper generalizes their result to linear schemes, providing (precise) intermediate bounds on the range, where the real cycle class map is an isomorphism. Moreover, we show that Lerbet's conjectured upper bound for the exponent of the cokernel of $H^i(X,\underline{I}^i(\mathcal{L})) \rightarrow H^i_\text{sing}(X(\RR),\ZZ(\mathcal{L}))$ cannot be improved.

This is part of the author's PhD thesis.

\end{abstract}

\maketitle              

\section{Introduction}

Let $X$ be a scheme, separated and finite type over a field $k$ admitting a real embedding. The (twisted) singular cohomology of the real points with the analytic topology $H^i_\text{sing}(X(\RR),\ZZ(\mathcal{L}))$ can be computed completely algebraically by computing the sheaf cohomology $H^i(X,\underline{I}^j(\mathcal{L}))$ for some power of the fundamental ideal sheaf $\underline{I}^j$. Jacobson \cite{Jacobson} constructs a real cycle class map \[H^i(X,\underline{I}^j(\mathcal{L})) \rightarrow H^i_\text{sing}(X(\RR),\ZZ(\mathcal{L}))\] and shows that this is an isomorphism for any scheme $X$ if $j\geq\dim(X)+1$ \cite[Corollary 8.11]{Jacobson} or \cite[Theorem 3.12]{RealCycleClassMap} for the twisted case. In \cite{RealCycleClassMap} it is further shown that the real cycle class map is an isomorphism already for $j\geq i$, if $X$ is a cellular scheme, i.e.\ admits a stratification by affine spaces $\AA^n_k$. The goal of this paper is to improve this bound and understand the non-bijectivity for linear schemes in the sense of Jannsen \cite{JannsenLinear} and Totaro \cite{TotaroLinear}, which include cellular schemes as a \enquote{$0$-linear case}\footnote{They will not be $0$-linear in the sense of the definition but behave like they were. More precisely, they are stratified by $0$-linear schemes and are covered by \Cref{lem: real cycle class map for stratifications by An Gm products}.}.

Jacobson's proof, that the real cycle class map 
\[H^i(X,\underline{I}^{\dim(X)+1}) \rightarrow H^i_\text{sing}(X(\RR),\ZZ)\] is an isomorphism, can be separated into two steps.
First, consider $\colim_j \underline{I}^j$ with transition maps $\langle\!\langle-1\rangle\!\rangle\colon \underline{I}^j\to \underline{I}^{j+1}$ and show $H^i(X,\colim_j\underline{I}^{j}) \cong H^i_\text{sing}(X(\RR),\ZZ)$. Secondly, show that on $X$ this colimit already stabilizes after finitely many steps, i.e.\ $\colim_j \underline{I}^j \cong \underline{I}^{\dim(X)+1}$, by showing that $\langle\!\langle-1\rangle\!\rangle\colon \underline{I}^j\to \underline{I}^{j+1}$ induces an isomorphism for $j\geq \dim(X)+1$.

In section 4 we proceed in the same way by showing that for smooth \enquote{$n$-linear} schemes we get an isomorphism (also for the twisted version), see \Cref{thm: real cycle class map is isomorphism for mult -1 Pister form isomorphism}
\[H^i(X,\underline{I}^{i+n}) \overset{\cong}{\longrightarrow} H^i(X,\underline{I}^{i+n+1}) \overset{\cong}{\longrightarrow}\dots \overset{\cong}{\longrightarrow}H^i(X,\underline{I}^{\dim(X)+1}) \overset{\cong}{\longrightarrow} H^i_\text{sing}(X(\RR),\ZZ).\]
This is the same factorization of the real cycle class maps that was used in \cite{RealCycleClassMap} to show the corresponding result for cellular schemes and $n=0$. The main tool to prove this statement is the inductive description of linear schemes in terms of closed immersions and open complements, together with the localization sequence for $I^j(\mathcal{L})$-cohomology.

We start with section 2 by providing the necessary background on $I^j(\mathcal{L})$-cohomology in terms of the Rost-Schmid complex. In section 3 linear schemes in the sense of Jannsen and Totaro are introduced. Section 4 provides the core study of $H^i(X,\underline{I}^{j}) \to H^i(X,\underline{I}^{j+1})$ for linear schemes. In section 5 we study explicit examples and show that the exponent of the cokernel of the real cycle class map for $\PP^c_\RR\times \mathbb{G}_m^{d-c}$ is exactly the conjectured upper bound in \cite[Conjecture 6.6]{Lerbet}, showing that it cannot be improved.
\subsection*{Acknowledgment}
I would like to thank Matthias Wendt and Marcus Zibrowius for their useful comments and feedback on an earlier draft.

This research was conducted in the framework of the research training group {\em GRK 2240: Algebro-Geometric Methods in Algebra, Arithmetic and Topology}.

\section{Background on \texorpdfstring{$I^j(\mathcal{L})$}{Ij(L)}-cohomology}
Given a field $F$ with $\chara(F)\neq 2$, the Grothendieck--Witt ring $\GW(F)$ is the group completion of the set of isomorphism classes of non-degenerate symmetric bilinear forms over $F$ with respect to orthogonal sums. We write $\langle a\rangle \in \GW(F)$ for the class of the symmetric form $\varphi_a\colon F^2 \to F, (x,y)\mapsto axy$ for $a \in F^\times$ and denote their orthogonal sums by $\langle a,b\rangle = \langle a\rangle +\langle b\rangle$. The Witt ring is defined as $\W(F) := \GW(F)/(\langle1\rangle + \langle -1\rangle)$, i.e.\ the quotient of the Grothendieck--Witt ring by the hyperbolic form. The fundamental ideal is defined as $I(F) := \ker (\rk\colon \W(F)\to \ZZ/2\ZZ)$ and is additively generated by Pfister forms $\langle\!\langle a\rangle\!\rangle := \langle a,-1\rangle$. The fundamental ideal powers form a filtration \[\W(F) \supseteq I(F) \supseteq I^2(F)\supseteq \dots\] and we set $\bar{I}^q:= I^q(F)/I^{q+1}(F)$ with the convention $I^q(F) = \W(F)$ for $q\leq 0$. 

Every non-degenerate symmetric bilinear form over $\RR$ is equivalent to one having only $1$ and $-1$ on the diagonal, so $\GW(\RR)\cong\ZZ[\{+1,-1\}]$, $\W(\RR)=\ZZ \supseteq  2^q\ZZ = I^q(\RR)$ and $\bar{I}^q(\RR) = \ZZ/2\ZZ$ for $q \geq 0$.

For $X$ a separated, finite type $k$-scheme and $\mathcal{L}$ a line bundle on $X$, the Rost-Schmid complex is defined in degree $i$ as \[C_{RS}(X,I^j(\mathcal{L}))_i := \bigoplus_{x\in X_{(i)}} I^{j+i}(\kappa(x), \det(\mathrm\Omega_{\kappa(x)/k})\otimes\mathcal{L}_{\kappa(x)})\]
with the differential defined in terms of residue maps. The sum is indexed by $X_{(i)}$, the points of dimension $i$. The groups appearing in those summands are all of the form \[I^q(F,L):= I^q(F)\otimes_{\ZZ[F^\times]} \ZZ[L\setminus\{0\}]\] for a one-dimensional $F$-vector space $L$. They are non-canonically isomorphic to $I^q(F)$.

Denote the homology of $C_{RS}(X,I^j(\mathcal{L}))$ by $H^{RS}_i(X,I^j(\mathcal{L}))$. This leads to the following identifications
\begin{align*}
    H^{RS}_i(X,I^j(\mathcal{L})) &= H^{BM}_i(X,\underline{I}^{\dim(X)+j}(\mathcal{L}))\\ &= H^{\dim(X)-i}(X,\underline{I}^{\dim(X)+j}(\mathcal{L}\otimes \omega_{X/k})),
\end{align*}
where the middle term is the Borel-Moore homology of the sheaf $\underline{I}^{\dim(X)+q}(\mathcal{L})$ and the latter identification, i.e.\ Poincaré duality, requires $X$ to be smooth. The Rost-Schmid complex here is used as a purely notational device. The main tool used in the paper is the localization sequence \[\dotsc \to H^{RS}_l(X, I^m(\mathcal{L})) \overset{j^\ast}{\longrightarrow} H^{RS}_l(U, I^m(\mathcal{L})) \longrightarrow  H^{RS}_{l-1}(Z, I^m(\mathcal{L}))\to \dotsc\] for a closed immersion $i\colon Z\hookrightarrow X$ with open complement $j\colon U\hookrightarrow X$ and line bundle $\mathcal{L}$ on $X$. In the Rost-Schmid notation the fundamental ideal power does not change and the (co)homological index has no \enquote{$-\operatorname{codim}_X(Z)$} for the group of $Z$.

\section{Linear schemes}

We will deal with two notions of linearity. The first, more general, J-linear notion is due to Jannsen \cite{JannsenLinear} and the more restrictive T-linear notion is due to Totaro \cite{TotaroLinear}.

\begin{definition}
	Let $Y$ be a scheme. Then $Y$ is called:
	\begin{description}[font=\normalfont]
		\item[$0$-J-linear] if $Y=\emptyset$ or $Y\cong\AA^N_k$ for some $N\geq 0$.
		\item[$n$-J-linear] if there exists a triple $(Z,X,U)$ of $k$-schemes with closed immersion $Z\hookrightarrow X$ and open complement $U=X\setminus Z$ such that:
		\begin{description}[font=\normalfont]
			\item[$Y=U$]\quad $Z$ and $X$ are $(n-1)$-J-linear or
			\item[$Y=X$]\quad $Z$ and $U$ are $(n-1)$-J-linear.
		\end{description}
	\end{description}
	A scheme $Y$ is called J-linear, if it is $n$-J-linear for some $n\geq 0$.
\end{definition}

\begin{remark}
    In the definition of J-linearity, non-regular closed immersion are permitted.

    Note that an $n$-J-linear scheme is also $m$-J-linear for any $m\geq n$.
\end{remark}

The first construction option above means that open complements of J-linear schemes in J-linear schemes are again J-linear. The second option is used to show that stratifications by J-linear schemes are J-linear.

We start by providing some constructions of linear schemes.

\begin{definition}
	A stratification of a scheme $X$ is a partition $X=\bigcup_{i\in I} U_i$, with $U_i\subseteq X$ locally closed, for which the following \enquote{boundary condition} holds:
	\begin{align*}
	U_i \cap \overline{U_j}\neq \emptyset\quad\Rightarrow\quad U_i\subseteq \overline{U_j}.
	\end{align*}
	A finite stratification is a stratification with finite index set $I$.
\end{definition}

\begin{remark}
    Two observations will help in the following proof.
    \begin{compactenum}[(i)]
        \item $\overline{U_j}=\bigcup_{U_i\subseteq\overline{U_j}} U_i$
        \item there is at most one stratum $U_j$ with $X=\overline{U_j}$.
    \end{compactenum}
    To see that (ii) holds, assume there are two strata $U$ and $V$ whose closure is $X$. Then $U$ and $V$ are open in $X$, as they are locally closed. So $X\setminus U\subsetneq X$ is a closed set containing $V$. This is a contradiction as $X=\overline{V}\subseteq X\setminus{U}\subsetneq X$. Hence there can only be one.
\end{remark}

\begin{lemma}\label{lem stratification by J-linear is J_linear}
	Let $X$ be a scheme admitting a finite stratification by locally closed J-linear schemes, then $X$ is J-linear.
\end{lemma}

\begin{proof}
	Write the stratification as $X=\bigcup_{i=1}^n U_i$ with $n=|I|$ and $U_i$ locally closed J-linear. Proceed by induction on the number of strata $n$.
	\begin{description}[font=\normalfont]
		\item[$n=1$] $X=U_1$ is J-linear by assumption.
		\item[$n\leadsto n+1$] In the stratification $X=\bigcup_{i=1}^{n+1} U_i$ pick one $U_j$ with $X\setminus \overline{U_j} \neq\emptyset$, which exists by previous observation and $n+1\geq 2$. Using the other observation write
		\[X = \overline{U_j} \cup (X\setminus \overline{U_j}) = \left(\bigcup_{U_i\subseteq\overline{U_j}} U_i\right) \cup \left(\bigcup_{U_i\nsubseteq\overline{U_j}} U_i\right).\]
		Now both parts are J-linear by induction hypothesis, since they are stratified by at most $n$ strata. Thus $X$ is J-linear as the disjoint union of a closed J-linear scheme and its J-linear open complement.
	\end{description}
	This concludes the proof.
\end{proof}

Instead of picking an arbitrary $U_j$ with $X\setminus \overline{U_j} \neq\emptyset$ in the previous proof, it is possible to pick a closed stratum. For this start with an arbitrary $U_j$ and consider the closure $\overline{U_j}=\bigcup_{U_i\subseteq\overline{U_j}} U_i$. Repeating the process with any $U_i\subseteq\overline{U_j}\setminus U_j$ will terminate after finitely many steps, because in each step there are strictly less strata available and there are only finitely many to begin with. For the last stratum we necessarily have $\overline{U_i}= U_i$.

\begin{lemma}\label{lem nice union of J-linear is J-linear}
	Let $X$ be a scheme and $A_1$, $\dots$, $A_n$ closed irreducible J-linear subschemes of $X$ with $\bigcap_{i\in I} A_i$ irreducible and J-linear for every subset $I\subseteq \{1,\dots,n\}$, then $\bigcup_{i=1}^n A_i$ is J-linear.
\end{lemma}

\begin{proof}
	Proceed by induction on $n$. The idea is to find an appropriate stratification of the union $\bigcup_{i=1}^n A_i$ by splitting it into $2^n-1$ parts, according to which of the $A_i$ an element belongs to.
	\begin{description}
		\item[\rm{$n=1$}] $A_1$ is J-linear by assumption. 
		\item[\rm{$n\leadsto n+1$}] Write
		\[\bigcup_{i=1}^{n+1} A_i = \bigcup_{i=1}^{n+1} \bigcup_{\substack{J\subseteq \{1,\dots,n+1\}, \\ |J|=i}} \left(\bigcap_{j\in J} A_j\right)\setminus \left(\bigcup_{j\notin J} A_j \cap \bigcap_{j\in J} A_j\right).\]
		Using the J-linearity of the closed subschemes $\bigcup_{j\notin J} \left(A_j \cap \bigcap_{j'\in J} A_{j'}\right)$ (by induction hypotheses, since $|J^c|\leq n$) and $\bigcap_{j\in J} A_j$ (by assumption) shows that the strata \[\left(\bigcap_{j\in J} A_j\right)\setminus \left(\bigcup_{j\notin J} A_j \cap \bigcap_{j\in J} A_j \right)\] are all J-linear.
	\end{description}
	
	To see that this actually is a stratification, the boundary condition needs to be checked. Picking a stratum means picking $U=\left(\bigcap_{j\in J} A_j\right)\setminus \left(\bigcup_{j\notin J} A_j \cap \bigcap_{j\in J} A_j\right)$ for some non-empty $J\subseteq\{1,\dots,n\}$. The stratum $U$ is open in the closed and irreducible  
	$\bigcap_{j\in J} A_j$, hence this is the closure of $U$. 
	
	\textbf{Claim:} Strata $U'$ (corresponding to a $J'$) meeting $\overline{U}=\bigcap_{j\in J} A_j$ satisfy $J\subseteq J'$.
    
	Suppose this is not the case and there is a $k\in J\setminus J'$. There cannot be an element both in $\overline{U} = \bigcap_{j\in J} A_j\subseteq A_k$ and
    \[U'=\left(\bigcap_{j\in J'} A_j\right)\setminus \left(\bigcup_{j\notin J'} A_j \cap \bigcap_{j\in J} A_j\right)\subseteq A^c_k.\]
	
	But $J\subseteq J'$ implies the wanted boundary condition \[U' \subseteq \bigcap_{j\in J'} A_j \subseteq \bigcap_{j\in J} A_j= \overline{U}.\]
\end{proof}

\begin{remark}
	Taking $n=3$ in the previous lemma gives the decomposition of the classical Venn diagram in its $7$ parts. Explicitly the stratification will be:
	\begin{compactenum}
		\item[] $A_1\cap A_2 \cap A_3$,
		\item[] $(A_1\cap A_2)\setminus A_3$,\quad $(A_1\cap A_3)\setminus A_2$,\quad $(A_2\cap A_3)\setminus A_1$,
		\item[] $A_1\setminus (A_2\cup A_3)$,\quad $A_2\setminus (A_1\cup A_3)$,\quad $A_3\setminus (A_1\cup A_2)$.
	\end{compactenum}
	The need for induction comes from the fact that for J-linearity the set differences of strata are written slightly different as (e.g.\ for $A_1\setminus (A_2\cup A_3)$):
	\[A_1 \setminus \left( \left( \left(A_1 \cap A_2\right)\setminus \left(A_1\cap A_2 \cap A_3\right) \right) \bigcup \left( \left(A_1 \cap A_3\right)\setminus \left(A_1\cap A_2 \cap A_3\right) \right) \bigcup \left(A_1\cap A_2 \cap A_3\right)\right).\]
\end{remark}

\begin{lemma}\label{lem: J-linearity of product}
	Let $X$ be $n$-J-linear and $Y$ be $m$-J-linear, then $X\times Y$ is $(n+m)$-J-linear.
\end{lemma}	

\begin{proof}
	Let $X$ be $n$-J-linear and $Y$ be $m$-J-linear. The idea is to iteratively build the product, since the construction steps for J-linear schemes are stable under taking a product with a fixed scheme. For $(n,m)=(0,0)$ we have $X\cong\AA_k^{\dim(X)}$ and $Y\cong \AA_k^{\dim(Y)}$, so their product is $X\times Y \cong \AA_k^{\dim(X) + \dim(Y)}$ and hence $0$-J-linear.
	
	The claim is that $X\times Y$ is $(n+m)$-J-linear. For $X$ there exists a triple $(Z,X',U)$ with a closed immersion $Z \hookrightarrow X'$ and open complement $U= X'\setminus Z$, such that:
	\begin{description}[font=\normalfont]
		\item[$X=U$]\quad $Z$ and $X'$ are $(n-1)$-J-linear or
		\item[$X=X'$]\quad $Z$ and $U$ are $(n-1)$-J-linear,
	\end{description} 
	by definition of $n$-J-linearity. By taking products with $Y$ we get a triple with a closed immersion $Z\times Y \hookrightarrow X'\times Y$ and open complement $U\times Y = (X'\times Y)\setminus (Z\times Y)$, such that:
	\begin{description}[font=\normalfont]
		\item[$X\times Y=U\times Y$]\quad $Z\times Y$ and $X'\times Y$ are $(n-1+m)$-J-linear or
		\item[$X\times Y=X'\times Y$]\quad $Z\times Y$ and $U\times Y$ are $(n-1+m)$-J-linear.
	\end{description} 
	This shows the induction step (the respective statement for $Y$ follows by symmetry).
\end{proof}

\begin{corollary}\label{cor: A^n x G_m^d cells are d-J-linear}
	The schemes $\AA^n_k\times\mathbb{G}_m^d$ are $d$-J-linear.
\end{corollary}

\begin{proof}
	The scheme $\AA^n_k$ is $0$-J-linear by definition and $\mathbb{G}_m=\AA^1_k\setminus\{0\}$ is $1$-J-linear.
\end{proof}

We will see in section 4 that this bound is optimal, i.e.\ $\AA^n_k\times\mathbb{G}_m^d$ is not $(d-1)$-J-linear.

The following variant of linear schemes comes from \cite{TotaroLinear}. For this notion of linearity, we provide a complete answer to when the real cycle class map is an isomorphism. 

\begin{definition}
	Let $Y$ be a scheme. Then $Y$ is called:
	\begin{description}[font=\normalfont]
		\item[$0$-T-linear] if $Y=\emptyset$ or $Y\cong\AA^N_k$ for some $N\geq 0$.
		\item[$n$-T-linear] if one of the following holds:
		\begin{compactenum}[(i)]
			\item there exists a triple $(Z,\AA^n_k,Y)$ of $k$-schemes with closed immersion $Z\hookrightarrow \AA^n_k$ and open complement $Y=\AA^N_k\setminus Z$, where $Z$ is $(n-1)$-T-linear
			\item $Y$ is stratified by $(n-1)$-T-linear schemes.
		\end{compactenum}
	\end{description}
	A scheme $Y$ is called T-linear, if it is $n$-T-linear for some $n\geq 0$.
\end{definition}

\begin{remark}
	The difference between J-linear and T-linear schemes is that J-linearity is stable under open complements of linear schemes in arbitrary linear schemes, not just affine spaces $\AA^n_k$ as it is for T-linearity. \Cref{lem stratification by J-linear is J_linear} shows that all T-linear schemes are J-linear. The author does not know any J-linear scheme that is not already T-linear. 
\end{remark}

\section{The real cycle class isomorphism}

Most statements in this section are direct analogs of the corresponding statements in \cite[Chapter 5]{RealCycleClassMap}, with changed bounds. To understand for which power of fundamental ideal the real cycle class map becomes an isomorphism, we need to understand in which range the multiplication $\langle\!\langle -1\rangle\!\rangle: \underline{I}^q\to\underline{I}^{q+1}$ induces an isomorphism on the cohomology groups of $X$. We start by investigating the situation for $\underline{\bar{I}}^q$ and lift the results to $\underline{I}^q$.

\begin{lemma}\label{lem: J-linear schemes have mult -1 Pfister form isomorphism on I bar}
	Let $X$ be a $n$-J-linear scheme over $\RR$, then the multiplication \[\langle\!\langle-1\rangle\!\rangle\colon H^{RS}_i(X,\bar{I}^j)\to H^{RS}_i(X,\bar{I}^{j+1})\] is an isomorphism for all $i \geq - j + n$ and injective for $i=-j+n-1$.
\end{lemma}

\begin{proof}
	Proceed by induction on $n$. For $n=0$ we have $X\cong \AA_\RR^{\dim(X)}$. By $\AA^1$-invariance both sides vanish for $i > 0$ and for $i = 0$ the map becomes
	\[2=\langle\!\langle-1\rangle\!\rangle\colon\quad 2^j\ZZ = \bar{I}^j(\RR) \cong H^{RS}_0(X,\bar{I}^j)\to H^{RS}_0(X,\bar{I}^{j+1}) \cong \bar{I}^{j+1}(\RR) = 2^{j+1}\ZZ, \]
	which is an isomorphism.
	
	Suppose we have a closed immersion $Z\hookrightarrow X'$ with open complement $U=X'\setminus Z$ then the localization sequence becomes
	$$
	\begin{tikzcd}[column sep = 0.95em]
	H^{RS}_{i+1}(U,\bar{I}^j) \arrow[r]\arrow[d, "\langle\!\langle-1\rangle\!\rangle"] & H^{RS}_{i}(Z,\bar{I}^j) \arrow[r]\arrow[d, "\langle\!\langle-1\rangle\!\rangle"] & H^{RS}_{i}(X',\bar{I}^j) \arrow[r]\arrow[d, "\langle\!\langle-1\rangle\!\rangle"] & H^{RS}_{i}(U,\bar{I}^j) \arrow[r]\arrow[d, "\langle\!\langle-1\rangle\!\rangle"] & H^{RS}_{i-1}(Z,\bar{I}^j)\arrow[d, "\langle\!\langle-1\rangle\!\rangle"]\\
	H^{RS}_{i+1}(U,\bar{I}^{j+1}) \arrow[r] & H^{RS}_{i}(Z,\bar{I}^{j+1}) \arrow[r] & H^{RS}_{i}(X',\bar{I}^{j+1}) \arrow[r] & H^{RS}_{i}(U,\bar{I}^{j+1}) \arrow[r] & H^{RS}_{i-1}(Z,\bar{I}^{j+1}),
	\end{tikzcd}
	$$
	where the diagram commutes as a diagram of $\W(\RR)$-modules. 
	
	If $X=X'$, and $Z$,$U$ are $(n-1)$-J-linear the first, second and fourth vertical morphisms are isomorphisms ($i \geq -j + n - 1$) and the fifth is injective ($i-1\geq -j + n -2 $) by induction. Therefore, the morphism in the middle is an isomorphism by the 5-lemma for all $i\geq-j+n-1$.
	
	For the other case of $X=U$, and $Z$,$X'$ are $(n-1)$-J-linear the part of the localization sequence is 
	$$
	\begin{tikzcd}[column sep = 0.83em]
	H^{RS}_{i}(Z,\bar{I}^j) \arrow[r]\arrow[d, "\langle\!\langle-1\rangle\!\rangle"] & H^{RS}_{i}(X',\bar{I}^j) \arrow[r]\arrow[d, "\langle\!\langle-1\rangle\!\rangle"] & H^{RS}_{i}(U,\bar{I}^j) \arrow[r]\arrow[d, "\langle\!\langle-1\rangle\!\rangle"] & H^{RS}_{i-1}(Z,\bar{I}^j) \arrow[r]\arrow[d, "\langle\!\langle-1\rangle\!\rangle"] & H^{RS}_{i-1}(X',\bar{I}^j)\arrow[d, "\langle\!\langle-1\rangle\!\rangle"]\\
	H^{RS}_{i}(Z,\bar{I}^{j+1}) \arrow[r] & H^{RS}_{i}(X',\bar{I}^{j+1}) \arrow[r] & H^{RS}_{i}(U,\bar{I}^{j+1}) \arrow[r] & H^{RS}_{i-1}(Z,\bar{I}^{j+1}) \arrow[r] & H^{RS}_{i-1}(X',\bar{I}^{j+1}).
	\end{tikzcd}
	$$
	For $i \geq - j + n$ the first, second, fourth and fifth vertical morphisms are isomorphisms and the 5-lemma gives the middle isomorphism again. For $i = - j + n - 1$ the first and second are still isomorphisms, but the fourth is in general just injective ($i-1 = -j + n -1$) so the 4-lemma shows, that the middle morphism is injective.
\end{proof}

\begin{remark}
	The possible failure of surjectivity can already be seen at the end of the localization sequence. The upper row vanishes one degree before the lower one.
	$$
	\begin{tikzcd}
	H^{RS}_{i}(Z,\bar{I}^{-i}) \arrow[r]\arrow[d, "\langle\!\langle-1\rangle\!\rangle"] & H^{RS}_{i}(X',\bar{I}^{-i}) \arrow[r]\arrow[d, "\langle\!\langle-1\rangle\!\rangle"] & H^{RS}_{i}(U,\bar{I}^{-i}) \arrow[r]\arrow[d, "\langle\!\langle-1\rangle\!\rangle"] & 0 \arrow[d, "\langle\!\langle-1\rangle\!\rangle"] \\
	H^{RS}_{i}(Z,\bar{I}^{-i+1}) \arrow[r] & H^{RS}_{i}(X',\bar{I}^{-i+1}) \arrow[r] & H^{RS}_{i}(U,\bar{I}^{-i+1}) \arrow[r] & H^{RS}_{i-1}(Z,\bar{I}^{-i+1})
	\end{tikzcd}
	$$
	For the use in a concrete example the bounds could be improved by keeping track of which construction steps are used in making the $n$-J-linear scheme. For example schemes with stratifications by affine spaces have the full range $i\geq -j$ where the multiplication $\langle\!\langle-1\rangle\!\rangle\colon H^{RS}_i(X,\bar{I}^j)\to H^{RS}_i(X,\bar{I}^{j+1})$ is an isomorphism.
\end{remark}

\begin{remark}
	The requirement to be defined over $\RR$ is somewhat necessary, because for general fields the $i=0$ case does not hold. For finite fields $\FF_q$ (or more generally fields in which every binary quadratic form is universal) we have $I(\FF_q) = \FF_q^\ast/(\FF_q^\ast)^2\cong\ZZ/2\ZZ$ and $I^2(\FF_q) = 0$, so there cannot be an isomorphism from $\bar{I}^1(\FF_q)\cong\ZZ/2\ZZ$ to $\bar{I}^2(\FF_q)=0$.
	
	The proof works for any field $F$ where the $i=0$ case holds, i.e.\ $\langle\!\langle -1 \rangle\!\rangle\colon \bar{I}^j(F)\to \bar{I}^{j+1}(F)$ is an isomorphism for $j\geq 0$.
\end{remark}

\begin{corollary}\label{cor: smooth J-linear schemes have mult -1 Pfister form isomorphism on I bar}
	Let $X$ be a  smooth, $n$-J-linear scheme over $\RR$, then the multiplication \[\langle\!\langle-1\rangle\!\rangle \colon H^i(X,\bar{I}^j)\to H^i(X,\bar{I}^{j+1})\] is an isomorphism for all $j\geq i + n$ and injective for $j=i + n -1$.
\end{corollary}

\begin{proof}
	For $X$ smooth we have $H^{RS}_{\dim(X)-i}(X,\bar{I}^{j-\dim(X)}) = H^{i}(C^{RS}(X,\bar{I}^j)) \cong H^i(X,\bar{I}^j)$ and therefore, 
	\begin{align*}
		H^i(X,\bar{I}^j) = H^{RS}_{\dim(X)-i}(X,\bar{I}^{j-\dim(X)}) \overset{\cong}{\longrightarrow} H^{RS}_{\dim(X)-i}(X,\bar{I}^{j-\dim(X)+1}) = H^i(X,\bar{I}^{j+1}),
	\end{align*}
	since $\dim(X)-i \geq -j+\dim(X) + n$, which holds by assumption $j\geq i + n$. The injectivity part holds for $\dim(X)-i = -j + \dim(X) + n -1$.
\end{proof}

\begin{corollary}\label{cor: stratified by J-linear schemes have mult -1 Pfister form isomorphism on I bar}
	Let $X$ be a  smooth scheme over $\RR$ which is stratified by $d$-J-linear schemes, then the multiplication \[\langle\!\langle-1\rangle\!\rangle \colon H^i(X,\bar{I}^j)\to H^i(X,\bar{I}^{j+1})\] is an isomorphism for all $j\geq i + d$ and injective for $j=i + d -1$.
\end{corollary}

 \begin{proof}
 	In the proof of \Cref{lem: J-linear schemes have mult -1 Pfister form isomorphism on I bar} the stratification step keeps the bounds, where the multiplication is an isomorphism, unchanged.
 \end{proof}

\begin{remark}\label{rem: at worst A^n x G_m^d cells}
	Using \Cref{cor: A^n x G_m^d cells are d-J-linear} together with \Cref{cor: stratified by J-linear schemes have mult -1 Pfister form isomorphism on I bar} shows that smooth schemes stratified by cells of the form $\AA^n_k\times \mathbb{G}_m^{d'}$ with all $d'\leq d$ satisfy the above property.
\end{remark}

In the more restrictive notion of T-linearity one can get a precise description of the influence the non-cellular part has.

\begin{lemma}\label{lem: T-linear analysis}
	Let $(Z,\AA^n_\RR,X)$ be a triple of $\RR$-schemes with closed immersion $Z\hookrightarrow \AA^n_\RR$ and open complement $X=\AA^n_\RR\setminus Z\neq\emptyset$, then the multiplication \[\langle\!\langle-1\rangle\!\rangle\colon H^{RS}_i(X,\bar{I}^j)\to H^{RS}_i(X,\bar{I}^{j+1})\] is an isomorphism for $(i,j)$ in the following cases:
	\begin{description}[font=\normalfont]
		\item[$i +j = 0$] if and only if $H^{RS}_{i-1}(Z,\bar{I}^{-i+1})=0$
		\item[$i+j\geq 1$] if and only if this holds for $(i-1,j)$ on $Z$ 
	\end{description}
\end{lemma}

\begin{proof}
	The case distinction is necessary because of the previously mentioned vanishing in the localization sequence. Note that $\dim(Z)\leq n-1$ (otherwise $Z=\AA^n_\RR$ and $X=\emptyset$). The localization sequence here is:
	$$
	\begin{tikzcd}[column sep = 0.73em]
	H^{RS}_{i}(Z,\bar{I}^j) \arrow[r]\arrow[d, "\langle\!\langle-1\rangle\!\rangle"] & H^{RS}_{i}(\AA^n_\RR,\bar{I}^j) \arrow[r]\arrow[d, "\langle\!\langle-1\rangle\!\rangle"] & H^{RS}_{i}(X,\bar{I}^j) \arrow[r]\arrow[d, "\langle\!\langle-1\rangle\!\rangle"] & H^{RS}_{i-1}(Z,\bar{I}^j) \arrow[r]\arrow[d, "\langle\!\langle-1\rangle\!\rangle"] & H^{RS}_{i-1}(\AA^n_\RR,\bar{I}^j)\arrow[d, "\langle\!\langle-1\rangle\!\rangle"]\\
	H^{RS}_{i}(Z,\bar{I}^{j+1}) \arrow[r] & H^{RS}_{i}(\AA^n_\RR,\bar{I}^{j+1}) \arrow[r] & H^{RS}_{i}(X,\bar{I}^{j+1}) \arrow[r] & H^{RS}_{i-1}(Z,\bar{I}^{j+1}) \arrow[r] & H^{RS}_{i-1}(\AA^n_\RR,\bar{I}^{j+1}).
	\end{tikzcd}
	$$
	\begin{description}[font=\normalfont]
		\item[$n = i = -j$] The diagram becomes 
		$$
		\begin{tikzcd}
		0 \arrow[r]\arrow[d] & \bar{I}^0(\RR) \arrow[r]\arrow[d, "\langle\!\langle-1\rangle\!\rangle"] & H^{RS}_{n}(X,\bar{I}^{-n}) \arrow[r]\arrow[d, "\langle\!\langle-1\rangle\!\rangle"] & 0 \arrow[r]\arrow[d] & 0\arrow[d]\\
		0 \arrow[r] & \bar{I}^1(\RR) \arrow[r] & H^{RS}_{n}(X,\bar{I}^{-n+1}) \arrow[r] & H^{RS}_{n-1}(Z,\bar{I}^{n+1}) \arrow[r] & 0,
		\end{tikzcd}
		$$
		by $\dim(Z)<n$, homotopy invariance and the vanishing of $H^{RS}$ in that range. The second vertical map is an isomorphism, so by the five lemma the middle morphism is an isomorphism if and only if the fourth one is.
		\item[$n-1\geq i = -j$] The diagram becomes 
		$$
		\begin{tikzcd}
		\ast \arrow[r]\arrow[d] & 0 \arrow[r]\arrow[d] & H^{RS}_{i}(X,\bar{I}^{-i}) \arrow[r]\arrow[d, "\langle\!\langle-1\rangle\!\rangle"] & 0 \arrow[r]\arrow[d] & 0\arrow[d]\\
		\ast \arrow[r] & 0 \arrow[r] & H^{RS}_{i}(X,\bar{I}^{-i+1}) \arrow[r] & H^{RS}_{i-1}(Z,\bar{I}^{i+1}) \arrow[r] & 0,
		\end{tikzcd}
		$$
		by the same reasoning. Now we have $H^{RS}_{i}(X,\bar{I}^{-i})\cong 0$ from the upper exact sequence and $H^{RS}_{i}(X,\bar{I}^{-i+1}) \cong H^{RS}_{i-1}(Z,\bar{I}^{i+1})$ from the lower one. Therefore, the middle map is an isomorphism if and only if $ H^{RS}_{i-1}(Z,\bar{I}^{i+1}) \cong 0$.
		\item[$i+j\geq 1,\ i=n$] The diagram becomes 
		$$
		\begin{tikzcd}
		0 \arrow[r]\arrow[d] & \bar{I}^{n+j}(\RR) \arrow[r]\arrow[d, "\langle\!\langle-1\rangle\!\rangle"] & H^{RS}_{n}(X,\bar{I}^j) \arrow[r]\arrow[d, "\langle\!\langle-1\rangle\!\rangle"] & H^{RS}_{n-1}(Z,\bar{I}^j) \arrow[r]\arrow[d, "\langle\!\langle-1\rangle\!\rangle"] & 0\arrow[d]\\
		0 \arrow[r] & \bar{I}^{n+j+1}(\RR) \arrow[r] & H^{RS}_{n}(X,\bar{I}^{j+1}) \arrow[r] & H^{RS}_{n-1}(Z,\bar{I}^{j+1}) \arrow[r] & 0,
		\end{tikzcd}
		$$
		by the same reasoning. The second map is an isomorphism, so the third map is an isomorphism if and only if the fourth is.
		\item[$i+j\geq 1,\ i\leq n-1$] The diagram becomes 
		$$
		\begin{tikzcd}
		\ast \arrow[r]\arrow[d] & 0 \arrow[r]\arrow[d] & H^{RS}_{i}(X,\bar{I}^j) \arrow[r]\arrow[d, "\langle\!\langle-1\rangle\!\rangle"] & H^{RS}_{i-1}(Z,\bar{I}^j) \arrow[r]\arrow[d, "\langle\!\langle-1\rangle\!\rangle"] & 0 \arrow[d]\\
		\ast \arrow[r] & 0 \arrow[r] & H^{RS}_{i}(X,\bar{I}^{j+1}) \arrow[r] & H^{RS}_{i-1}(Z,\bar{I}^{j+1}) \arrow[r] & 0,
		\end{tikzcd}
		$$
		again by the same reasoning. Now we have isomorphisms $H^{RS}_{i}(X,\bar{I}^j) \cong H^{RS}_{i-1}(Z,\bar{I}^j)$ and $H^{RS}_{i}(X,\bar{I}^{j+1}) \cong H^{RS}_{i-1}(Z,\bar{I}^{j+1})$, showing the assertion in the last case.
	\end{description}
	This finishes the proof.
\end{proof}

\begin{remark}
	Assuming $X$ and $Z$ to be smooth in the previous lemma means that \[\langle\!\langle-1\rangle\!\rangle\colon H^i(X,\bar{I}^i)\to H^i(X,\bar{I}^{i+1})\] is an isomorphism if and only if $H^{i-c-1}(Z,\bar{I}^{i-c+1})=0$ for $c=\codim_{\AA^n_\RR}(Z)=n-\dim(Z)$.
\end{remark}

\begin{proposition}\label{prop: mult -1 Pister form isomorphism on I for mult -1 Pister form isomorphism on I bar}
	Let $X$ be a smooth scheme over $\RR$ for which the multiplication \[\langle\!\langle-1\rangle\!\rangle \colon H^i(X,\bar{I}^j)\to H^i(X,\bar{I}^{j+1})\] is an isomorphism for all $j\geq i+n$, then
	\[\langle\!\langle-1\rangle\!\rangle \colon H^i(X,I^j(\mathcal{L}))\to H^i(X,I^{j+1}(\mathcal{L}))\] is an isomorphism for all $j\geq i+n$ and all line bundles $\mathcal{L}$ on $X$.
\end{proposition}

The $n=0$ case is treated in \cite[Proposition 5.5]{RealCycleClassMap} and this proof is verbatim the same. We recall the proof for convenience. 

\begin{proof}
    Consider the long exact sequences associated to $0\to \underline{I}^{j}\to \underline{I}^{j-1}\to \underline{\bar{I}}^{j-1}\to 0$. 
	$$
	\begin{tikzcd}[column sep = 0.58em]
	H^{i-1}(X,\underline{\bar{I}}^{j-1}) \arrow[r]\arrow[d, "\langle\!\langle-1\rangle\!\rangle"] & H^{i}(X,\underline{I}^j(\mathcal{L})) \arrow[r]\arrow[d, "\langle\!\langle-1\rangle\!\rangle"] & H^{i}(X,\underline{I}^{j-1}(\mathcal{L})) \arrow[r]\arrow[d, "\langle\!\langle-1\rangle\!\rangle"] & H^{i}(X,\underline{\bar{I}}^{j-1}) \arrow[r]\arrow[d, "\langle\!\langle-1\rangle\!\rangle"] & H^{i+1}(X,\underline{I}^j(\mathcal{L}))\arrow[d, "\langle\!\langle-1\rangle\!\rangle"]\\
	H^{i-1}(X,\underline{\bar{I}}^{j}) \arrow[r]& H^{i}(X,\underline{I}^{j+1}(\mathcal{L})) \arrow[r]& H^{i}(X,\underline{I}^{j}(\mathcal{L})) \arrow[r] & H^{i}(X,\underline{\bar{I}}^{j}) \arrow[r]& H^{i+1}(X,\underline{I}^{j+1}(\mathcal{L}))
	\end{tikzcd}
	$$
    The second and fifth vertical morphism is an isomorphism for $j\geq \dim(X)+1$ by \cite[Corollary 8.11]{Jacobson}. Now the bijectivity of the third vertical morphism follows from the fourth by downward induction and the five-lemma.
\end{proof}

\begin{remark}\label{rem: non-surj of Ibar}
    The previous proof also shows that if $\langle\!\langle-1\rangle\!\rangle \colon H^i(X,\underline{\bar{I}}^j)\to H^i(X,\underline{\bar{I}}^{j+1})$ is non-surjective for $j=i+n-1$, then $\langle\!\langle-1\rangle\!\rangle \colon H^i(X,\underline{I}^j)\to H^i(X,\underline{I}^{j+1})$ is also non-surjective, by the four-lemma.
\end{remark}

\begin{theorem}\label{thm: real cycle class map is isomorphism for mult -1 Pister form isomorphism}
	Let $X$ be a smooth $n$-J-linear scheme over $\RR$ and $\mathcal{L}$ a line bundle, then the real cycle class map
	\[cl^i_j\colon H^i(X,\underline{I}^j(\mathcal{L})) \to H_{\text{sing}}^i(X(\RR),\ZZ(\mathcal{L}))\] 
    \begin{compactenum}
        \item[] $j\geq i+n$: is an isomorphism
        \item[] $j \leq i+n-1$: is a morphism to a subgroup containing $2^{i+n-j}H_{\text{sing}}^i(X(\RR),\ZZ(\mathcal{L}))$.
    \end{compactenum}
    Moreover, it is injective for $j = i+n-1$ and for $j < i$ we have $\im(cl^i_j) = 2^{i-j}\im(cl^i_i)$.
\end{theorem}

\begin{proof}
	The proof is again the analog of \cite[Theorem 5.7]{RealCycleClassMap}. The theorem follows from \Cref{cor: smooth J-linear schemes have mult -1 Pfister form isomorphism on I bar}, and \Cref{prop: mult -1 Pister form isomorphism on I for mult -1 Pister form isomorphism on I bar}, together with the factorization of the real cycle class maps 
    \[H^i(X,\underline{I}^{i+n-1}(\mathcal{L})) \hookrightarrow H^i(X,\underline{I}^{i+n}(\mathcal{L})) \overset{\cong}{\rightarrow}\dots \overset{\cong}{\rightarrow}H^i(X,\underline{I}^{\dim(X)+1}(\mathcal{L})) \overset{\cong}{\rightarrow} H^i_\text{sing}(X(\RR),\ZZ(\mathcal{L})).\]
    The description of the images follows from the commutative diagram
    $$
    \begin{tikzcd}
        H^i(X,\underline{I}^{j+1}(\mathcal{L})) \arrow[r]\arrow[d, swap, "cl^i_{j+1}"] & H^i(X,\underline{I}^j(\mathcal{L}))\arrow[d, "cl^i_j"]\\
        H^i_\text{sing}(X(\RR),\ZZ(\mathcal{L})) \arrow[r, swap, "\cdot 2"] & H^i_\text{sing}(X(\RR),\ZZ(\mathcal{L}))
    \end{tikzcd}
    $$
    which shows $\im(cl^i_j)\supseteq 2\im(cl^i_{j+1})$, see also \cite[Lemma 4.1]{Lerbet}.
\end{proof}

\begin{remark}
	The formulation of \Cref{thm: real cycle class map is isomorphism for mult -1 Pister form isomorphism} should not be seen as optimal. All the statements in this chapter give a more detailed description of how the sequence of maps \[H^i(X,\underline{I}^{i}) \longrightarrow H^i(X,\underline{I}^{i+1}) \longrightarrow\dots \longrightarrow H^i(X,\underline{I}^{i+n}) \overset{\cong}{\longrightarrow} H^i_\text{sing}(X(\RR),\ZZ)\] behaves and provide possible improvements in concrete situations. 
    In particular, \Cref{lem: T-linear analysis} and \Cref{rem: non-surj of Ibar} can show non-bijectivity of the real cycle class map.
\end{remark}

For example, an application of \Cref{rem: at worst A^n x G_m^d cells} together with the previous line of arguments leads to the following corollary.

\begin{corollary}\label{lem: real cycle class map for stratifications by An Gm products}
    Let $X$ be a smooth $\RR$-scheme stratified by cells of the form $\AA^n_k\times \mathbb{G}_m^{d'}$ with all $d'\leq d$ and varying $n$, then the real cycle class map
	\[H^i(X,\underline{I}^j(\mathcal{L})) \to H_{\text{sing}}^i(X(\RR),\ZZ(\mathcal{L}))\]
	is an isomorphism for $j\geq i+d$ and injective for $j=i+d-1$.
\end{corollary}

\begin{proof}
    This is the combination of \Cref{cor: A^n x G_m^d cells are d-J-linear}, \Cref{cor: stratified by J-linear schemes have mult -1 Pfister form isomorphism on I bar}, \Cref{prop: mult -1 Pister form isomorphism on I for mult -1 Pister form isomorphism on I bar} and the argument in \Cref{thm: real cycle class map is isomorphism for mult -1 Pister form isomorphism}.
\end{proof}

\section{Examples}

The easiest example is $\AA^n_k$ because it is $0$-linear and cellular. There are no non-trivial line bundles on $\AA^n_k$, so there are no twist to consider, and by $\AA^1$-invariance the cohomology is concentrated in degree $0$. There we have $H^0(\AA^n_k,\underline{I}^j)=I^j(k)$ and the sequence becomes
\[\underbrace{H^0(\AA^n_\RR,\underline{I}^{-1})}_{=\W(\RR)=\ZZ} \overset{\cdot 2}{\rightarrow} \underbrace{H^0(\AA^n_\RR,\underline{I}^{0})}_{=\W(\RR)=\ZZ} \overset{\cong}{\rightarrow}\underbrace{H^0(\AA^n_\RR,\underline{I}^{1})}_{=I(\RR)=2\ZZ} \overset{\cong}{\rightarrow}\dots \overset{\cong}{\rightarrow}\underbrace{H^0(\AA^n_k,\underline{I}^{n+1})}_{=I^{n+1}(\RR)=2^{n+1}\ZZ} \overset{\cong}{\rightarrow} \underbrace{H^0_\text{sing}(\RR^n,\ZZ)}_{\cong \ZZ}.\]

The next example is $\mathbb{G}_m$, which is $1$-linear but not $0$-linear. Again, there are no non-trivial line bundles and the cohomology is concentrated in degree $0$. We have \[H^0(\mathbb{G}_m,\underline{I}^j)=I^j(k)\oplus \langle\!\langle t\rangle\!\rangle I^{j-1}(k)\subseteq I^j(k(t)).\]
The sequence becomes
\[\underbrace{H^0(\mathbb{G}_m,\underline{I}^{-1})}_{=\W(\RR)\oplus \langle\!\langle t\rangle\!\rangle\W(\RR)} \!\overset{\cdot 2}{\rightarrow}\!\! \underbrace{H^0(\mathbb{G}_m,\underline{I}^{0})}_{=\W(\RR)\oplus \langle\!\langle t\rangle\!\rangle\W(\RR)} \!\!\overset{(\operatorname{id},\cdot 2)}{\longrightarrow}\!\!\underbrace{H^0(\mathbb{G}_m,\underline{I}^{1})}_{=I(\RR)\oplus \langle\!\langle t\rangle\!\rangle\W(\RR)} \!\overset{\cong}{\rightarrow}\!\underbrace{H^0(\mathbb{G}_m,\underline{I}^{2})}_{=I^{2}(\RR)\oplus \langle\!\langle t\rangle\!\rangle I(\RR)}\cong \underbrace{H^0_\text{sing}(\RR\setminus \{0\},\ZZ)}_{\cong \ZZ^2}.\]
Here the multiplication by $\langle\!\langle -1\rangle\!\rangle\colon H^0(\mathbb{G}_m,\underline{I}^{j}) \to H^0(\mathbb{G}_m,\underline{I}^{j+1})$ is not an isomorphism for $j=0$, as it would be for cellular schemes, but for $j\geq 1$. This improved bound, compared to Jacobson's $j\geq 2 = \dim(\mathbb{G}_m)+1$, is also explained by \cite[Proposition 3.4]{JensRealCycleClassMap}.

The same argument is applicable to $\AA^n_k\times \mathbb{G}_m^d$ with 
\[H^0(\AA^n_k\times \mathbb{G}_m^d, \underline{I}^j) \cong \bigoplus_{i=0}^d \left(I^{j-i}(k)\right)^{\binom{d}{i}},\]
showing that $\langle\!\langle -1\rangle\!\rangle\colon H^0(\AA^n_k\times \mathbb{G}_m^d,\underline{I}^{j}) \to H^0(\AA^n_k\times \mathbb{G}_m^d,\underline{I}^{j+1})$ is an isomorphism for $j\geq d$ and injective but not surjective for $j < d$.

The cokernel of $H^0(\mathbb{G}_m^d,\underline{I}^{0}) \to H^0(\mathbb{G}_m^d,\underline{I}^{d})\cong H^0_\text{sing}((\RR\setminus \{0\})^d,\ZZ)$ has exponent $2^d$.

This can be generalized as follows:
\begin{proposition}
    For any integers $d > c \geq 0$ the cokernel of 
    \[H^c(\PP^c_\RR\times \mathbb{G}_m^{d-c},\underline{I}^c(\mathcal{\mathcal{O}\rm{(c+1)}})) \rightarrow H^c_\text{sing}(\RP^c\times(\RR\setminus\{0\})^{d-c},\ZZ(\mathcal{O}(c+1)))\] has exponent $2^{d-c}$.
\end{proposition}

\begin{proof}
    The $\RR$-scheme $\PP^c_\RR\times \mathbb{G}_m^{d-c}$ is cellular in the sense of \cite{MorelSawant}, induced by the standard cellular structure of $\PP^n_k$. The cells are all of the form $\AA^i_k\times\mathbb{G}_m^{d-c}$ for $i=0,\dots,n$. The cellular complex \cite{MorelSawant},\cite{PhDThesis} for $\underline{I}^c(\mathcal{O}(c+1))$ coefficients around degree $c$ is:
    $$
	\begin{tikzcd}
	\bigoplus_{i=0}^{d-c} \left(I^{1-i}(k)\right)^{\binom{d-c}{i}}\arrow[r, "\partial^{c-1}"] & \bigoplus_{i=0}^{d-c} \left(I^{-i}(k)\right)^{\binom{d-c}{i}}\arrow[r] & 0
	\end{tikzcd}
	$$
    By \cite{PhDThesis}, the differential $\partial^{c-1}$ is zero for the twist by $\mathcal{O}(c+1)$. Hence 
    \[H^c(\PP^c_\RR\times \mathbb{G}_m^{d-c},\underline{I}^c(\mathcal{\mathcal{O}\rm{(c+1)}})) = \bigoplus_{i=0}^{d-c} \left(I^{-i}(k)\right)^{\binom{d-c}{i}}\]
    and the same argument as before shows that the cokernel of the real cycle class map has exponent $2^{d-c}$.
\end{proof}

This solves the non-improvability of the conjectured upper bound on the exponent in \cite[Conjecture 6.6]{Lerbet}.

\end{document}